\documentclass[12pt]{amsart}

\usepackage{amsmath,amssymb,amscd,amsthm,graphicx,enumerate}

\usepackage{hyperref}
\hypersetup{breaklinks=true}
\usepackage{graphicx,float}

\numberwithin{equation}{section}
\theoremstyle{plain}

\makeatletter
\newtheorem*{rep@theorem}{\rep@title}
\newcommand{\newreptheorem}[2]{%
\newenvironment{rep#1}[1]{%
 \def\rep@title{#2 \ref{##1}}%
 \begin{rep@theorem}}%
 {\end{rep@theorem}}}
\makeatother

\newtheorem{theorem}[equation]{Theorem}
\newreptheorem{theorem}{Theorem}
\newtheorem{thm}[equation]{Theorem}
\newtheorem{proposition}[equation]{Proposition}

\newtheorem{corollary}[equation]{Corollary}
\newtheorem{conjecture}[equation]{Conjecture}
\newtheorem{openproblem}[equation]{Open problem}

\newtheorem{exercise}[equation]{Exercise}

\theoremstyle{definition}
\newtheorem{definition}[equation]{Definition}

\newtheorem{defn}[equation]{Definition}

\newtheorem*{question*}{Question}

\newtheorem{example}[equation]{Example}
\newtheorem*{references}{References}
\newtheorem*{relatedPDEs}{Related PDEs}
\newcommand{\RR}{\mathbb{R}}
\newcommand{\Grad}{\nabla}
\newcommand{\abs}[1]{\lvert#1\rvert}

\newcommand{\norm}[1]{\lVert#1\rVert}

\newcommand{\M}{{\mathcal M}}

\newcommand{\R}{\mathbb R}

\newcommand{\dist}{\operatorname{dist}}

\newcommand{\al}{\alpha}

\newcommand{\Lap}{\Delta}

\newcommand{\eps}{\varepsilon}

\newcommand{\ra}{\rightarrow}

\providecommand{\abs}[1]{\lvert #1\rvert}
\providecommand{\norm}[1]{\lvert\lvert #1\rvert\rvert}

\def\XXint#1#2#3{{\setbox0=\hbox{$#1{#2#3}{\int}$}
     \vcenter{\hbox{$#2#3$}}\kern-.5\wd0}}

\begin{document}

\title[Mean curvature flow of surfaces]{Lectures on mean curvature flow of surfaces}
\author{Robert Haslhofer}

\thanks{I thank the organizers of the summer schools at UT Austin in 2021 and at CRM Montreal in 2024, and all the participants for their questions and feedback. Some of the results described in these lectures are from joint work with Kyeongsu Choi, Wenkui Du, Or Hershkovits, Bruce Kleiner and Brian White. My research has been supported by an NSERC Discovery Grant and a Sloan Research Fellowship.}

\date{\today}
\maketitle


\begin{abstract}
Mean curvature flow is the most natural evolution equation in extrinsic geometry, and shares many features with Hamilton's Ricci flow from intrinsic geometry. In this lecture series, I will provide an introduction to the mean curvature flow of surfaces, with a focus on the analysis of singularities. We will see that the surfaces evolve uniquely through neck singularities and nonuniquely through conical singularities. Studying these questions, we will also learn many general concepts and methods, such as monotonicity formulas, epsilon-regularity, weak solutions, and blowup analysis that are of great importance in the analysis of a wide range of partial differential equations. These lecture notes are from summer schools at UT Austin and CRM Montreal, and also contain a detailed discussion of open problems and conjectures.
\end{abstract}

\tableofcontents

\section{Overview and basic properties}

In this first lecture, I will give a quick informal introduction to the mean curvature flow of surfaces.\footnote{For concreteness we focus on $2$-dimensional surfaces in $\mathbb{R}^3$, but of course many things could be generalized to higher dimensions and other ambient spaces.}

A smooth family of embedded surfaces $\{M_t\subset \R^{3}\}_{t\in I}$ moves by mean curvature flow if
\begin{equation}\label{eq_mcf}
\partial_t x = \vec{H}(x)
\end{equation}
for $x\in M_t$ and $t\in I$. Here, $I\subset \R$ is an interval, $\partial_t x$ is the normal velocity at $x$, and $\vec{H}(x)$ is the mean curvature vector at $x$.

If we write $\vec{H}=H\vec{\nu}$, where $\vec{\nu}$ is the inwards unit normal, then $H$ is given by the sum of the principal curvatures, $H=\kappa_1+\kappa_2$. Recall that the principal curvatures are the eigenvalues of the second fundamental form. More concretely, given any point $p$ on a surface $M$, if we express the surface locally as a graph of a function $u$ over the tangent space $T_pM$, then $\kappa_1(p)$ and $\kappa_2(p)$ are simply the eigenvalues of $\mathrm{Hess}(u)(p)$.

\begin{example}[Shrinking sphere and cylinder] If $M_t=S^2({r(t)})$ is a round sphere, then equation \eqref{eq_mcf} reduces to an ODE for the radius, namely
$\dot{r}=-2/r$.
The solution with $r(0)=R$ is $r(t)=\sqrt{R^2-4t}$, where $t\in (-\infty,R^2/4)$.
Similarly, we have the round shrinking cylinder $M_t=\R \times S^1({r(t)})$ with $r(t)=\sqrt{R^2-2t}$, where $t\in (-\infty,R^2/2)$.
\end{example}

\begin{exercise}[Graphical evolution]
Show that if $M_t =\mathrm{graph}(u(\cdot,t))$ is the graph of a time-dependent function $u(\cdot,t):\mathbb{R}^2\to\mathbb{R}$, then\footnote{Hint: It is important to remember that $\partial_t x$ denotes the \emph{normal} velocity at $x$.}
\begin{equation}
\partial_t u= \sqrt{1+|D u|^2}\,\mathrm{div}\!\left(\frac{Du}{\sqrt{1+|Du|^2}} \right) \, .
\end{equation}
\end{exercise}

Instead of viewing the mean curvature flow as an evolution equation for the hypersurfaces $M_t$, we can also view it as an evolution equation for a smooth family of embeddings $X: M^2 \times I \rightarrow \RR^{3}$ with $M_t=X(M,t)$.
Setting $x=X(p,t)$, equation \eqref{eq_mcf}  then takes the form
\begin{equation}\label{eq_mcf2}
\partial_t X(p,t) = \Lap_{M_t} X(p,t).
\end{equation}

The fundamental idea of geometric flows is to deform a given geometric object into a nicer one, by evolving it by a heat-type equation.
This indeed works very well, as illustrated by the following theorem.

\begin{theorem}[Huisken's convergence theorem]
Let $M_0\subset \R^{3}$ be a closed embedded surface. If $M_0$ is convex, then the mean curvature flow $\{M_t\}_{t\in[0,T)}$ starting at $M_0$ converges to a round point.
\end{theorem}

The convex case ($\kappa_1\geq 0$ and $\kappa_2\geq 0$) is of course very special.
In more general situations, we encounter the formation of singularities.

\begin{example}[Neckpinch singularity]
If $M_0$ has the topology of a sphere but the geometry of a dumbbell, then the neck pinches off. As blowup limit we get a selfsimilarly shrinking round cylinder. There is also a degenerate variant of this example, where as blowup limit along suitable tip points one gets the selfsimilarly translating bowl soliton.
\end{example}

In the study of mean curvature flow (and indeed of most nonlinear PDEs) it is of crucial importance to understand singularities:
\begin{itemize}
\item How do singularities look like?
\item Can we continue the flow through singularities?
\item What is the size and the structure of the singular set?
\item Is the evolution through singularities unique or nonunique?
\end{itemize}

The analysis of singularities will be our main focus for the following lectures. We conclude this first lecture, by summarizing a few basic properties of the mean curvature flow.

First, by standard parabolic theory, given any compact initial hypersurface $M_0\subset \R^{3}$ (say smooth and embedded), there exists a unique smooth solution $\{M_t\}_{t\in [0,T)}$ of \eqref{eq_mcf} starting at $M_0$, and defined on a maximal time interval $[0,T)$. The maximal time $T$ is characterized by the property that the curvature blows up, i.e.
$\lim_{t\to T}\max_{M_t}\abs{A}=\infty$.

Second, like for any second order parabolic equation, time scales like distance squared. For example, if $u(x,t)$ solves the heat equation $\partial_t u =\Delta u$, then given any $\lambda>0$ the parabolically rescaled function $u^\lambda(x,t)=u(\lambda x,\lambda^2 t)$ again solves $\partial_t u^\lambda=\Delta u^\lambda$. The following exercise shows that the same rescaling indeed works for mean curvature flow:

\begin{exercise}[Parabolic rescaling]
Let $M_t\subset \R^{3}$ be a mean curvature flow of surfaces, and let $\lambda>0$.
Let  $M_{t'}^\lambda$ be the family of surfaces obtained by the parabolic rescaling $x'=\lambda x$, $t'=\lambda^2 t$, i.e. let $M_{t'}^\lambda=\lambda M_{\lambda^{-2}t'}$. Show that $M_{t'}^\lambda$ indeed solves \eqref{eq_mcf}.
\end{exercise}

In particular, due to the scaling, all estimates naturally take place in parabolic balls $P(x_0,t_0,r)=B(x_0,r)\times (t_0-r^2,t_0]$.

Third, by the maximum principle (aka avoidance principle) mean curvature flows do not bump into each other. More precisely, if $M_t$ and $N_t$ are two  mean curvature flows (say at least one of them compact), then $\dist(M_t,N_t)$ is nondecreasing in time. In particular, if $M_{t_0}$ and $N_{t_0}$ are disjoint, then so are $M_t$ and $N_t$ for all $t\geq t_0$. Similarly, the flow does not bump into itself, i.e. embeddedness is also preserved.

Forth, the evolution equation \eqref{eq_mcf} implies evolution equations for the induced metric $g_{ij}$, the area element $d\mu$, the normal vector $\vec{\nu}$, the mean curvature $H$, and the second fundamental form $A$:

\begin{proposition}[Evolution equations for geometric quantities]\label{prop_evol_eq}
If $M_t\subset \R^{3}$ evolves by mean curvature flow, then 
\begin{equation}\label{eq_evol}
\begin{array}{lll}
\partial_t g_{ij}=-2HA_{ij} & \partial_td\mu=-H^2 d\mu &\partial_t\vec{\nu}=-\nabla H\\
\partial_t H=\Lap H+\abs{A}^2 H & \partial_t A^i_j=\Lap A^i_j+\abs{A}^2 A^i_j.&
\end{array}
\end{equation}
\end{proposition}
For example, the evolution of $g_{ij}=\partial_i X\cdot \partial_j X$ is computed via
\begin{equation}
\partial_t g_{ij}=2\partial_i (H\vec{\nu})\cdot \partial_j X=2H\partial_i \vec{\nu}\cdot \partial_j X=-2HA_{ij}.
\end{equation}

\begin{exercise}[Evolution equations]
Show that if $G=G(t)$ is a smooth family of invertible matrices, then $\tfrac{d}{dt} \ln\det G=\textrm{tr}_G{\tfrac{d}{dt}G}$. Use this to derive the evolution equation for $d\mu=\sqrt{\det g_{ij}}d^2x$. Moreover, derive the remaining evolution equations by differentiating $A_{ij}=\partial_i\partial_j X \cdot \nu$ (Hint: simplify using Simons identity) and taking the trace.
\end{exercise}

In particular, if $M_0$ is compact the total area decreases according to
\begin{equation}\label{eq_areamon}
\frac{d}{dt}\textrm{Area}(M_t)=-\int_{M_t} H^2 d\mu.
\end{equation}

One can think of this as a variational characterization of the mean curvature flow as the gradient flow of the area functional.

Finally, using Proposition \ref{prop_evol_eq} and the maximum principle we obtain:

\begin{corollary}[mean-convexity and convexity]
Let $M_t\subset \R^{3}$ be a mean curvature flow of closed surfaces.
If $H\geq 0$ at $t=0$, then $H\geq 0$ for all $t>0$.
Similarly, convexity is also preserved.
\end{corollary}

We note that convexity, i.e. $\kappa_1\geq 0$ and $\kappa_2\geq 0$, is a much stronger assumption than mean-convexity, i.e. $H=\kappa_1+\kappa_2\geq 0$.\footnote{Compare this with the study of convex functions, i.e. functions satisfying $\mathrm{Hess}\, u \geq 0$, versus subharmonic functions, i.e. functions satisfying $\Lap u\geq 0$.} Mean-convexity is on the one hand is flexible enough to allow for interesting singularities, e.g. the neck-pinch and degenerate neck-pinch, and flexible enough for interesting applications\footnote{The most famous application for inverse mean curvature flow is the proof of the Penrose inequality by Huisken-Ilmanen \cite{HuiskenIlmanen}. For some other applications of mean-convex flows see e.g. \cite{Schulze_isoperimetric,huisken-sinestrari3,BHH,HaslhoferKetover,WangZhou,HaslhoferKetover2,LiokumovichMaximo}.}, but on the other hand, as we will see, rigid enough to obtain a detailed description of singularities.

\bigskip

\begin{relatedPDEs}
The mean curvature flow is closely related to other geometric PDEs, including in particular the harmonic map flow $\partial_t u = \Delta_{M,N} u$, Hamilton's Ricci flow $\partial_t g_{ij} = -2\mathrm{Rc}_{ij}$ and the Yang-Mills flow $\partial_t A = -D_A^\ast F_A$. It is very fruitful to study these PDEs in parallel, as new insights on one of them often leads to progress on the others.
\end{relatedPDEs}

\bigskip

\section{Monotonicity formula and epsilon-regularity}

In this second lecture, we discuss Huisken's monotonicity formula and the epsilon-regularity theorem for the mean curvature flow.

Recall that by equation \eqref{eq_areamon} the total area is monotone under mean curvature flow. However, since $\textrm{Area}(\lambda M)=\lambda^2 \textrm{Area}(M)$, this is not that useful when considering blowup sequences with $\lambda\to \infty$. A great advance was made by Huisken, who discovered a scale invariant monotone quantity.
To describe this, let $\M=\{M_t\subset \R^{3}\}$ be a smooth mean curvature flow of surfaces, say with at most polynomial volume growth,
let $X_0=(x_0,t_0)$ be a point in space-time, and let
\begin{equation}
 \rho_{X_0}(x,t)=\frac{1}{4\pi(t_0-t)} e^{-\frac{\abs{x-x_0}^2}{4(t_0-t)}}\qquad (t<t_0)
\end{equation}
be the $2$-dimensional backwards heat kernel centered at $X_0$.

\begin{theorem}[Huisken's monotonicity formula]\label{thm_huisken_mon}
\begin{equation}\label{eq_huisken_mon}
 \frac{d}{dt}\int_{M_t} \rho_{X_0} d\mu = -\int_{M_t} \left|\vec{H}-\frac{(x-x_0)^\perp}{2(t-t_0)}\right|^2 \rho_{X_0} d\mu\qquad (t<t_0).
\end{equation}
\end{theorem} 

Huisken's monotonicity formula \eqref{eq_huisken_mon} can be thought of as weighted version of \eqref{eq_areamon}. A key property is its invariance under  rescaling.

\begin{exercise}[Scaling invariance]
Let $x'=\lambda(x-x_0)$, $t'=\lambda^2(t-t_0)$, and consider the rescaled flow $M^\lambda_{t'}=\lambda(M_{t_0+\lambda^{-2}t'}-x_0)$. Prove that
\begin{equation}
 \int_{M_t} \rho_{X_0}(x,t) \, d\mu_t(x) = \int_{M^\lambda_{t'}} \rho_{0}(x',t')\,  d\mu_{t'}(x')\qquad (t'<0).
\end{equation}
\end{exercise}

Another key property is that the equality case of \eqref{eq_huisken_mon} exactly characterizes the selfsimilarly shrinking solutions:

\begin{exercise}[Shrinkers]
Let $\{M_t\subset \R^{3}\}_{t\in (-\infty,0)}$ be an ancient solution of the mean curvature flow. Prove that
$\vec{H}-\frac{x^\perp}{2t}=0$ for all $t<0$ if and only if $M_t=\sqrt{-t}M_{-1}$ for all $t<0$.
\end{exercise}

\begin{proof}[Proof of Theorem \ref{thm_huisken_mon}]
We may assume without loss of generality that $X_0=(0,0)$. The proof of Huisken's monotonicity formula essentially amounts to deriving belows pointwise identity \eqref{eq_pointwise} for $\rho=\rho_0$.

Since the tangential gradient of $\rho$ is given by $\nabla^{M_t}\rho = D\rho-(D\rho\cdot\vec{\nu}) \vec{\nu}$,
the intrinsic Laplacian of $\rho$ can be expressed as
\begin{equation}
\Lap_{M_t}\rho = \textrm{div}_{M_t}\nabla^{M_t}\rho
=\textrm{div}_{M_t}D\rho+\vec{H}\cdot D\rho.
\end{equation}
Observing also that $\tfrac{d}{dt}\rho=\partial_t \rho+\vec{H}\cdot D\rho$, we compute
\begin{align}
(\tfrac{d}{dt}+\Lap_{M_t})\rho&=\partial_t \rho+\textrm{div}_{M_t}D\rho+2\vec{H}\cdot D\rho\nonumber\\
&=\partial_t \rho+\textrm{div}_{M_t}D\rho+\frac{\abs{\nabla^\perp \rho}^2}{\rho}-\Big|\vec{H}-\frac{\nabla^\perp \rho}{\rho}\Big|^2 \rho+H^2\rho.
\end{align}
We can now easily check that $\partial_t \rho+\textrm{div}_{M_t}D\rho+\frac{\abs{\nabla^\perp \rho}^2}{\rho}=0$. Thus 
\begin{equation}\label{eq_pointwise}
(\tfrac{d}{dt}+\Lap_{M_t}-H^2)\rho=-\Big|\vec{H}-\frac{x^\perp}{2t}\Big|^2 \rho.
\end{equation}
Using also the evolution equation $\tfrac{d}{dt}d\mu =-H^2d\mu $, we conclude that
\begin{equation}
 \frac{d}{dt}\int_{M_t} \rho\, d\mu = -\int_{M_t} \left|\vec{H}-\frac{x^\perp}{2t}\right|^2 \rho\,  d\mu  \qquad (t<0).
\end{equation}
This proves the theorem.
\end{proof}

More generally, if $M_t$ is only defined locally, say in $B(x_0,\sqrt{8}\rho)\times (t_0-\rho^2,t_0)$, then we can localize using the cutoff function
\begin{equation}
\chi^\rho_{X_0}(x,t)=\left(1-\frac{\abs{x-x_0}^2+4(t-t_0)}{\rho^2}\right)_+^3\, .
\end{equation}
Observing that $(\tfrac{d}{dt}-\Lap_{M_t})\chi^\rho_{X_0}\leq 0$ we still get the monotonicity inequality
\begin{equation}\label{app_loc_mon}
 \frac{d}{dt}\int_{M_t} \rho_{X_0}\chi^\rho_{X_0} d\mu \leq -\int_{M_t} \left|\vec{H}-\frac{(x-x_0)^\perp}{2(t-t_0)}\right|^2 \rho_{X_0}\chi^\rho_{X_0} d\mu.
\end{equation}
The monotone quantity appearing on the left hand side,
\begin{equation}
\Theta^\rho(\M,X_0,r):=\int_{M_{t_0-r^2}} \rho_{X_0}\chi_{X_0}^\rho d\mu,
\end{equation}
is called the Gaussian density ratio. Note that $\Theta^\infty(\M,X_0,r)\equiv 1$ for all $r>0$ if and only if $\M$ is a multiplicity one plane containing $X_0$.

We will now discuss the epsilon-regularity theorem for the mean curvature flow, which gives definite curvature bounds in a neighborhood of definite size, provided the Gaussian density ratio is close to one.

\begin{theorem}[epsilon-regularity]\label{app_thm_easy_brakke}
There exist universal constants $\eps>0$ and $C<\infty$ with the following significance.
If $\M$ is a smooth mean curvature flow in a parabolic ball $P(X_0,8\rho)$ with
\begin{equation}
 \sup_{X\in P(X_0,r)}\Theta^{\rho}(\M,X,r)<1+\eps
\end{equation}
for some $r\in(0,\rho)$, then
\begin{equation}
 \sup_{P(X_0,r/2)}\abs{A}\leq {C}r^{-1}.
\end{equation}
\end{theorem}

Note also that if $\Theta<1+\frac{\eps}{2}$ holds at some point and some scale, then $\Theta<1+\eps$ holds at all nearby points and somewhat smaller scales.

\begin{proof}
 Suppose the assertion fails. Then there exist a sequence of smooth flows $\M^j$ in $P(0,8 \rho_j)$, for some $\rho_j> 1$, with
\begin{equation}
 \sup_{X\in P(0,1)}\Theta^{\rho_j}(\M^j,X,1)<1+j^{-1},
\end{equation}
but such that there are points $X_j\in P(0,1/2)$ with $\abs{A}(X_j)> j$.

Using the so-called point selection technique, we can find space-time points $Y_j\in P(0,3/4)$ with $Q_j=\abs{A}(Y_j)> j$ such that
\begin{equation}\label{app_brakke_point_sel}
 \sup_{P(Y_j,j/10Q_j)}\abs{A}\leq 2 Q_j.
\end{equation}
Let us explain how the point selection works: Fix $j$. If $Y^0_j=X_j$ already satisfies (\ref{app_brakke_point_sel}) with $Q^0_j=\abs{A}(Y^0_j)$, we are done. Otherwise, there is a point $Y^1_j\in P(Y_j^0,j/10Q^0_j)$ with $Q^1_j=\abs{A}(Y^1_j)>2Q^0_j$.
If $Y^1_j$ satisfies (\ref{app_brakke_point_sel}), we are done. Otherwise, there is a point $Y^2_j\in P(Y_j^1,j/10Q^1_j)$ with $Q^2_j=\abs{A}(Y^2_j)>2Q^1_j$, etc.
Note that $\frac{1}{2}+\frac{j}{10Q_j^0}(1+\frac{1}{2}+\frac{1}{4}+\ldots)<\frac{3}{4}$. By smoothness, the iteration terminates after a finite number of steps, and the last point of the iteration lies in $P(0,3/4)$ and satisfies (\ref{app_brakke_point_sel}).

Continuing the proof of the theorem, let $\hat\M^j$ be the flows obtained by shifting $Y_j$ to the origin and parabolically rescaling by $Q_j=\abs{A}(Y_j)\to\infty$.
Since the rescaled flow satisfies $\abs{A}(0)=1$ and $\sup_{P(0,j/10)}\abs{A}\leq 2$, we can pass smoothly to a nonflat global limit. On the other hand, by the rigidity case of (\ref{app_loc_mon}), and since
\begin{equation}
 \Theta^{\hat\rho_j}(\hat\M^j,0,Q_j)<1+j^{-1},
\end{equation}
where $\hat\rho_j=Q_j\rho_j\to\infty$, the limit is a flat plane; a contradiction.
\end{proof}

\bigskip

\begin{relatedPDEs}
Monotonicity formulas and epsilon-regularity theorems are a key tool in the study of many PDEs. Historically, this goes back at least to the classical monotonicity formula for harmonic functions. Let me mention a few further instances. The monotonicity formula and epsilon-regularity theorem for the harmonic map flow are due to Struwe \cite{Struwe_mon}.
The elliptic cousin of Huisken's monotonicity formula is the monotonicity formula for minimal surfaces, which states that if $M$ satisfies $H=0$ then the function $r\mapsto \frac{\textrm{Area}(M\cap B_r(p_0))}{\pi r^2}$ is monotone. The epsilon-regularity theorem for minimal surfaces was proved by Allard \cite{Allard}. Similarly, for manifolds with nonnegative Ricci-curvature the volume ratios are monotone (but in the opposite direction) by a result of Bishop-Gromov \cite{Gromov_mon}. A monotonicity formula for the Ricci flow was discovered by Perelman \cite{Perelman}. The epsilon-regularity theorem for Einstein metrics was proved by Anderson \cite{Anderson} and the one for the Ricci flow by Hein-Naber \cite{HeinNaber}.

\end{relatedPDEs}

\bigskip

\section{Noncollapsing, curvature and convexity estimate}

In this third lecture, we discuss the noncollapsing result of Andrews, as well as the local curvature estimate and the convexity estimate.

The following quantitative notion of embeddedness plays a key role in the theory of mean-convex mean curvature flow.

\begin{defn}[noncollapsing]\label{def_andrews_static}
A closed embedded mean-convex surface $M\subset \RR^{3}$ is called \emph{$\alpha$-noncollapsed},
if each point $p\in M$ admits interior and exterior balls tangent at $p$ of radius $\alpha/H(p)$.
\end{defn}

By compactness, every closed embedded mean-convex initial surface is $\alpha$-noncollapsed for some $\alpha>0$. This is preserved under the flow:

\begin{thm}[Andrews' noncollapsing theorem]\label{thm_andrews}
 If the initial surface $M_0\subset \mathbb{R}^3$ is $\alpha$-noncollapsed, then so is $M_t$ for all $t\in [0,T)$.
\end{thm}

\begin{proof}[Proof (sketch)]
For $x\in M$, the interior ball of radius $r(x)=\alpha/H(x)$ has the center point $c(x)=x+r(x)\nu(x)$.
The condition that this is indeed an interior ball is equivalent to the inequality
\begin{equation}\label{inequ_ymincx}
 \norm{y-c(x)}^2 \geq r(x)^2 \qquad \textrm{for all $y\in M$}.
\end{equation}
Observing $\norm{y-c(x)}^2=\norm{y-x}^2-2r(x)\langle y-x,\nu(x)\rangle +r(x)^2$ and inserting $r(x)=\alpha/H(x)$, the inequality (\ref{inequ_ymincx}) can be rewritten as
\begin{equation}
 \frac{2\langle y-x,\nu(x)\rangle}{\norm{y-x}^2} \leq \frac{H(x)}{\alpha} \qquad \textrm{for all $y\in M$}.
\end{equation}
Now, given a mean-convex mean curvature flow $M_t=X(M,t)$ of closed embedded surfaces, we consider the quantity
\begin{equation}\label{def_overlZ}
 Z^\ast(x,t)=\sup_{y\neq x} \frac{2\langle X(y,t)-X(x,t),\nu(x,t)\rangle}{\norm{X(y,t)-X(x,t)}^2}.
\end{equation}
A rather lengthy computation, which we skip, yields that
\begin{equation}\label{evol_ineq}
 \partial_t Z^\ast\leq \Lap Z^\ast+\abs{A}^2Z^\ast
\end{equation}
in the viscosity sense. Together with the evolution equation for the mean curvature, $\partial_t H = \Delta H + |A|^2 H$, this implies
\begin{equation}
\partial_t \frac{Z^\ast}{H}
\leq \Lap \frac{Z^\ast}{H} + 2 \langle \Grad \log H, \Grad \frac{Z^\ast}{H} \rangle\, .
\end{equation}
Hence, by the maximum principle, if the inequality $Z^\ast/H\leq 1/\alpha$ holds for $t=0$, then it also holds for all $t>0$. This proves interior noncollapsing. Finally, a similar argument shows that the inequality
\begin{equation}\label{eq_extnoncoll}
 Z_\ast(x,t)=\inf_{y\neq x} \frac{2\langle X(y,t)-X(x,t),\nu(x,t)\rangle}{\norm{X(y,t)-X(x,t)}^2}\geq -\frac{H(x,t)}{\alpha}
\end{equation}
is also preserved, which yields exterior noncollapsing.
\end{proof}

The following estimate gives curvature control on a parabolic ball of definite size starting from a mean 
curvature bound at a single point.

\begin{theorem}[local curvature estimate]\label{thm-intro_local_curvature_bounds}
For all $\al>0$ there exist $\rho=\rho(\al)>0$ and $C=C(\al)<\infty$ with the following significance.
If $\M$ is an $\al$-noncollapsed flow defined in a parabolic ball $P(p,t,r)$ centered at a point 
$p\in M_t$ with   $H(p,t)\leq r^{-1}$, then  
\begin{equation}\label{eqn-intro_curvature_estimate}
 \sup_{P(p,t,\rho r)} H \leq C r^{-1}\, .
\end{equation}
\end{theorem}

For comparison, recall that if $u$ is a positive solution of an elliptic or parabolic partial differential equation, then by the classical Harnack estimate the values of $u$ at nearby points are comparable.

\begin{proof}
Suppose towards a contradiction that the assertion fails. Then, there is a sequence $\M^j$ of $\al$-noncollapsed flows defined in $P(0,0,j)$ with 
$H(0,0)\leq j^{-1}$, but such that
\begin{equation}\label{eq_curv_contr}
\sup_{P(0,0,1)}H\geq j.
\end{equation}
We can assume that the outward normal of $M^j_0$ at the origin
is $e_{3}$, and that for every $R<\infty$  the flows $\mathcal{M}^j$ foliate $B(0,R)$ for $j\geq j_0(R)$.

We claim that the sequence $\M^j$ converges in the
pointed Hausdorff sense to the static plane $\{x_3=0\}$ in $\R^{3}\times (-\infty,0]$. 
Indeed, 
for any $R<\infty$ and $d>0$ consider the closed ball $B_{R,d}\subset\{x_3\leq d\}$ of radius $R$ that touches the point $d e_3$.
When $R$ is large, 
it will take approximately time $R d$ for $B_{R,d}$ to leave the upper 
halfspace $\{x_{3}>0\}$.  Since $0\in M^j_0$ for all $j$, it follows
that $B_{R,d}$ cannot be contained in the interior of $M^j_t$
for any $t \in [-T,0]$, where $T \simeq R d$.   
Thus, for large $j$ 
we can
find $d_j\leq d$
such that $B_{R,{d_j}}$ has interior contact with $M^j_t$ at 
some point $q_j$, where $\langle q_j,e_{3}\rangle < d$, $\|q_j\|\lesssim \sqrt{Rd}$,
and 
$\liminf_{j\ra\infty} \langle q_j,e_{3}\rangle\geq 0$. Now, since $M^j_t$ satisfies 
the $\al$-noncollapsing condition, 
 there is a closed ball $ B_j$
with radius at least $\al R/2$ making  exterior contact with $M^j_0$
at $q_j$.
By a simple geometric calculation, this 
implies that $M^j_t$ has  height  $\lesssim d/\al$ in the ball
$B(0,R')$, where $R'$ is comparable  to $\sqrt{Rd}$.  As $d$ and $R$ are 
arbitrary, this implies that for any $T>0$,  and any compact subset $Y\subset\{x_{3}>0\}$, 
 for large $j$ the time slice
$M^j_t$ is disjoint from $Y$, for all 
$t \geq -T$.  Finally, observe that for any $T>0$ and any
compact subset $Y\subset \{x_{3}<0\}$, 
the time slice
$M^j_t$ contains $Y$ for all 
$t \in[-T, 0]$, and large $j$,
because $M^j_{-T}$ contains a ball whose forward evolution under
the flow contains $Y$ at any time $t\in [-T, 0]$. This proves the claim.

Finishing the proof of the theorem, together with one-sided minimization (see below),
we infer that
 for every $\eps>0$, every $t\leq 0$ and every ball $B(x,r)$ centered
on the hyperplane $\{x_{3}=0\}$ we have
\begin{equation}\label{eqn_densitybound}
\textrm{Area}(M_t^j \cap B(x,r)) \leq (1+\eps)\pi r^2\,,
\end{equation}
whenever $j$ is large
enough. Finally, applying the epsilon-regularity theorem for the mean curvature flow (Theorem \ref{app_thm_easy_brakke}) this yields
\begin{equation}
\lim\sup_{j\to\infty}\sup_{P(0,0,1)}\abs{A}=0\, .
\end{equation}
This contradicts \eqref{eq_curv_contr}, and thus concludes the proof.
\end{proof}

\begin{exercise}[One-sided minimization]
\label{rem_one_sided_minimization}
Use Stokes' theorem and mean-convexity to prove the density bound \eqref{eqn_densitybound}.
\end{exercise}

The next estimate gives pinching of the curvatures towards positive:

\begin{theorem}[Convexity estimate]\label{thm-intro_convexity_estimate}
For all $\eps>0$ and $\al >0$, there exists a constant $\eta=\eta(\eps,\al)<\infty$ with the following significance.
If $\M$ is an $\al $-noncollapsed flow defined in a parabolic ball $P(p,t,\eta\, r)$ centered at a point 
$p\in M_t$ with $H(p,t)\leq r^{-1}$, 
then
\begin{equation}
\kappa_1(p,t)\geq -\eps r^{-1}.
\end{equation} 
\end{theorem}

In particular, any ancient $\alpha$-noncollapsed flow $\{M_t\subset\mathbb{R}^3\}_{t\in (-\infty,T)}$, for example any blowup limit of an $\alpha$-noncollapsed flow, is convex.

\begin{proof}
Fixing $\al $,  let $\eps_0\leq 1/\al $ be the infimum of the epsilons for which the assertion holds, 
and suppose towards a contradiction that $\eps_0>0$. Then, there is a sequence $\M^j$ of $\al$-noncollapsed flows defined in 
$P(0,0,j)$, such that $(0,0)\in \M^j$ and $H(0,0)\leq 1$, but ${\kappa_1}(0,0)\to -\eps_0$ as $j\ra \infty$.
By Theorem \ref{thm-intro_local_curvature_bounds} (local curvature estimate), after passing to a subsequence, 
$\M^j$  converges smoothly to a limit $\M^\infty$ in $P(0,0,\rho/2)$. Observe that for $\M^\infty$ we have $\kappa_1(0,0)=-\eps_0$ and thus $H(0,0)=1$.

By continuity $H>1/2$
in  $P(0,0,r)$ for some $r\in (0,\rho/2)$.  Furthermore, we have $\kappa_1/H\geq -\eps_0$
everywhere in $P(0,0,r)$.  This is because   every $(p,t)\in \M^\infty\cap P(0,0,r)$
is a limit of points $(p_j,t_j)\in  \M^j$, and  for every $\eps > \eps_0$, if 
$\eta=\eta(\eps,\al)$, then for 
large $j$ enough $\M^j$ is defined in $P(p_j,t_j,\eta/ H(p_j,t_j))$,
which implies that $\kappa_1(p_j,t_j)\geq -\eps H(p_j,t_j)$. 
Thus, in $P(0,0,r)$  the ratio $\kappa_1/H$ attains a 
negative minimum $-\eps_0$ at $(0,0)$. Since $\kappa_1<0$ and $H>0$ the Gauss curvature $K=\kappa_1\kappa_2$ at the origin is strictly negative.
However, by the equality case of the maximum principle for $\kappa_1/H$, the surface locally splits as a product and thus the Gauss curvature must vanish; a contradiction.
\end{proof}

Pushing the above methods a bit further, via an induction on scale argument it can be shown that all blowup limits of $\alpha$-noncollapsed flows are smooth and convex until they become extinct. Together with the recent classification by Brendle-Choi, it then follows that for the flow of mean-convex embedded surfaces all singularities at the first singular time are modelled either by a round shrinking sphere, a round shrinking cylinder or a self-similarly translating bowl soliton. This makes precise the intuition that, unless the entire surface shrinks to a round point, all singularities look like neck-pinches or degenerate neck-pinches.

\bigskip

\begin{relatedPDEs}
The notion of $\alpha$-noncollapsing for the mean curvature flow is inspired by Perelman's $\kappa$-noncollapsing for the Ricci flow \cite{Perelman}. For example, $\alpha$-noncollapsing rules out blowup limits like $\mathbb{R}\times$Grim-Reaper, and $\kappa$-noncollapsing rules out blowup limits like $\mathbb{R}\times$Cigar. More generally, conditions with touching balls are used frequently to establish estimates for elliptic or parabolic PDEs. Also, as discussed, the curvature estimate is related, at least in spirit, to Harnack inequalities for positive solutions of elliptic or parabolic PDEs. Finally, for 3d Ricci flow there is the Hamilton-Ivey pinching estimate \cite{Hamilton_survey}.
\end{relatedPDEs}

\bigskip

\section{Weak solutions}

In this lecture, we discuss notions of weak (aka generalized) solutions that allow one to continue the evolution through singularities.

Recall that by the avoidance principle smooth mean curvature flows do not bump into each other. Motivated by this, a family of closed sets $\{C_t\}$ is called a \emph{subsolution} if it avoids all smooth solutions, namely
\begin{equation}\label{eq_avoid}
C_{t_0}\cap M_{t_0}=\emptyset \qquad \Rightarrow \qquad C_{t}\cap M_{t}=\emptyset \quad \forall t\in[t_0,t_1]\, ,
\end{equation}
whenever $\{M_t\}_{t\in[t_0,t_1]}$ is a smooth mean curvature flow of closed surfaces.

\begin{definition}[level-set flow] The \emph{level-set flow} $\{F_t(C)\}_{t\geq 0}$ of any closed set $C$ is the maximal subsolution $\{C_t\}_{t\geq 0}$ with $C_0=C$.
\end{definition}

\begin{proposition}[basic properties] The level-set flow is well-defined and unique, and has the following basic properties:
\begin{itemize}
\item semigroup property: $F_0(C)=C$, $F_{t+t'}(C)=F_t(F_{t'}(C))$.
\item commutes with translations:  $F_t(C+x)=F_t(C)+x$.
\item containment: if $C\subseteq C'$, then $F_t(C)\subseteq F_t(C')$.
\end{itemize}
\end{proposition}

\begin{proof} Observe first that by translation-invariance of smooth solutions, a family of closed sets $\{C_t\}$ is a subsolution if and only if
\begin{equation}\label{equiv_char}
d(C_t,M_t)\geq d(C_{t_0},M_{t_0})\quad \forall t\in[t_0,t_1]\, ,
\end{equation}
whenever $\{M_t\}_{t\in[t_0,t_1]}$ is a smooth closed mean curvature flow.
Now, considering the closure of the union of all subsolutions, namely
\begin{equation}
F_{t'}(C)=\overline{\bigcup \{ C_{t'} \, | \ \{C_t\}_{t\geq 0} \textrm{ is a subsolution} \}}\, ,
\end{equation}
we see that the level-set flow exists and is unique. Finally, the basic properties immediately follow from existence and uniqueness.
\end{proof}

Using the characterization \eqref{equiv_char} it is also not hard to see that level-set solutions are consistent with classical solutions, namely if $\{M\}_{t\in [0,T)}$ is a smooth mean curvature flow of closed surfaces, then
\begin{equation}
F_t(M)=M_t\qquad \forall t\in[0,T).
\end{equation}
Furthermore, by interposing a $C^{1}$-surface one can check that level-set flows also avoid each other, namely
\begin{equation}\label{eq_avoid_lev}
C\cap C' =\emptyset \qquad \Rightarrow\qquad F_t(C)\cap F_t(C')=\emptyset \quad \forall t\geq 0\, ,
\end{equation}
provided that at least one of $C,C'$ is compact. While the level-set solution is unique by definition, the evolution can be nonunique:

\begin{example}[fattening]
There exists a closed embedded surface $M\subset\mathbb{R}^3$, which looks like a wheel with many spokes, that encounters a conical singularity after which $F_t(M)$ develops nonempty interior.
\end{example}

For the sake of intuition, it helps to compare this with the nonuniqueness/fattening of the figure X under curve shortening flow. Now, to capture this nonuniqueness phenomenon in more detail, given any closed embedded surface $M\subset\mathbb{R}^3$ we denote by $K$ the compact domain enclosed by $M$, and set $K':=\overline{K^c}$. Note that $\partial K=M=\partial K'$. We then consider the space-time tracks of their level-set flows, namely
\begin{equation}
\mathcal{K}:=\{ (x,t)\in\mathbb{R}^3\times [0,\infty)\, | \, x\in F_t(K)\}\, ,
\end{equation}
and
\begin{equation}
\mathcal{K}':=\{ (x,t)\in\mathbb{R}^3\times [0,\infty)\, | \, x\in F_t(K')\}\, .
\end{equation}

\begin{definition}[outer and inner flow] The \emph{outer flow} is define by
\begin{equation}
M_t:= \{ x \in \mathbb{R}^3 \, | \, (x,t)\in \partial\mathcal{K} \}\, ,
\end{equation}
and the \emph{inner flow} is defined by
\begin{equation}
M_t':= \{ x \in \mathbb{R}^3 \, | \, (x,t)\in \partial\mathcal{K}' \}\, .
\end{equation}
\end{definition}
Here, for technical reasons it is most convenient to work with the boundary of space-time sets, but alternatively one can check that
\begin{equation}
M_{t}=\lim_{t'\nearrow t} \partial F_{t'}(K)\, ,
\end{equation}
and similarly for the inner flow.

\begin{definition}[discrepancy time] The \emph{discrepancy time} is 
\begin{equation}
T_{\textrm{disc}}:= \inf\{\,  t>0\, | \, M_t\neq M_t'\, \} \in (0,\infty]\, .
\end{equation}
\end{definition}
This captures the first time when nonuniqueness happens.\footnote{As shown recently in \cite{BK_mult1} the discrepancy time $T_{\textrm{disc}}$ is in fact equal to the fattening time $T_{\textrm{fat}}:=  \inf\{\,  t>0\, | \, \textrm{Int}(F_t(M)) \neq 0 \} $.}

While level-set solutions are very well suited for discussing the question of uniqueness versus nonuniqueness, we also need another notion of solutions, so called Brakke flows, that is better suited for arguments based on the monotonicity formula and for passing to limits.

Recall that a Radon measure $\mu$ in $\mathbb{R}^{3}$ is integer two-recifiable, if at almost every point it possess a tangent plane of integer multiplicity. Namely, setting $\mu_{x,\lambda}(A)=\lambda^{-2}\mu(\lambda A+x)$, for $\mu$-a.e. $x$ we have
\begin{equation}
\lim_{\lambda\to 0} \mu_{x,\lambda} = \theta \mathcal{H}^2\lfloor P\, ,
\end{equation}
for some positive integer $\theta$ and some plane $P$. We write $P=T_x\mu$. Also recall that the associated integral varifold is defined by
\begin{equation}
V_{\mu}(\psi)= \int \psi(x,T_x\mu) \, d\mu(x)\, .
\end{equation}

\begin{definition}[Brakke flows]
A two-dimensional \emph{integral Brakke flow} in $\mathbb{R}^{3}$ is a family of Radon measures $\mathcal M = \{\mu_t\}_{t\in I}$ in $\mathbb{R}^{3}$ that is integer two-rectifiable for almost every time and satisfies
\begin{equation}\label{Brakke_inequ}
\frac{d}{dt} \int \varphi \, d\mu_t \leq \int \left( -\varphi {\vec{H}}^2 + D\varphi \cdot \vec{H} \right)\, d\mu_t
\end{equation}
for all test functions $\varphi\in C^1_c(\mathbb{R}^{3},\mathbb{R}_+)$. Here, $\tfrac{d}{dt}$ denotes the limsup of difference quotients, and ${\vec{H}}$ denotes the mean curvature vector of the associated varifold $V_{\mu_t}$.\footnote{By convention, the right hand side is interpreted as $-\infty$ whenever it does not make sense literally. Hence, it actually makes sense literally at almost every time.}

\end{definition}
The definition is of course motivated by fact that for smooth solutions \eqref{Brakke_inequ} would hold as equality. In general though, only the inequality is preserved under passing to weak limits. All integral Brakke flows that we encounter in this lecture series are
\begin{itemize}
\item \emph{unit-regular},  i.e. near every space-time point of Gaussian density $1$ the flow is regular in a two-sided parabolic ball, and
\item \emph{cyclic}, i.e. for a.e. $t$ the $\mathbb{Z}_2$ flat chain $[V_{\mu_t}]$ satisfies $\partial [V_{\mu_t}]=0$.
\end{itemize}

Intuitively, the last item simply means that we can color the inside and outside, which in particular rules out blowup limits like Y$\times \mathbb{R}$.

Also, if $M_t$ is any smooth mean curvature flow, then $\mu_t:=\mathcal{H}^2\lfloor M_t$ is of course a unit-regular, cyclic, integral Brakke flow. 

\begin{theorem}[compactness]\label{thm_comp}
Any sequence of integral Brakke flows $\mu^i_t$  with uniform area bounds on compact subsets has a subsequence $\mu^{i'}_t$ that converges to an integral Brakke flow $\mu_t$, namely (i) for every $t$ we have $\mu^{i'}_t\to \mu_t$ as Radon measures, and (ii) for a.e. $t$ after passing to a further subsequence $i''=i''(t)$ we have $V_{\mu^{i''}_t}\to V_{\mu_t}$ as varifolds.
Moreover, if the sequence is unit-regular/cyclic, then so is the limit.
\end{theorem}

\begin{proof}[Proof (sketch)] Using \eqref{Brakke_inequ} and the Peter-Paul inequality
\begin{equation}\label{eq_peter_paul}
 D\varphi \cdot \vec{H}\leq \frac{1}{2}\frac{|D\varphi|^2}{\varphi} + \frac{1}{2}\varphi |\vec{H}|^2\ ,
 \end{equation}
 we see that for every $\varphi\in C^2_c(\mathbb{R}^3;\mathbb{R}_+)$ there exists $C(\varphi)<\infty$ such that
 \begin{equation}
 L^i_{\varphi}(t):=\int \varphi d\mu_t^i - C(\varphi) t
 \end{equation}
 is decreasing in $t$. Hence, after passing to a subsequence, that may depend on $\varphi$, the functions $ L^i_{\varphi}(t)$ converge to a monotone function $L(t)$. In particular, $\int \varphi d\mu_t^i$ has a limit for all $t$. Repeating this process for a countable dense subset of $C^2_c(\mathbb{R}^3;\mathbb{R}_+)$ we can arrange that
 \begin{equation}
 \mu^i_t\to \mu_t
 \end{equation}
as Radon measures for all $t$. Now, by the assumed area bounds, and the inequalities \eqref{Brakke_inequ} and \eqref{eq_peter_paul} we have
\begin{equation}\label{mass_bounds}
\sup_{t\in [t_0,t_1]}\mu_t(K) + \int_{t_0}^{t_1} \int_K H^2 \, d\mu_t^i\, dt \leq C(K,t_1,t_2)
\end{equation}
for every compact set $K\subset\mathbb{R}^3$. Hence, by Allard's compactness theorem, for a.e. $t$ we can find a subsequence $i(t)$ such that
\begin{equation}
V_{\mu_t^{i(t)}}\to V_{\mu_t}
\end{equation}
as integral varifolds. Furthermore, by Fatou's lemma we have
\begin{equation}
\int  \varphi |\vec{H}|^2 \, d\mu_t \leq \liminf_{i\to \infty} \int \varphi |\vec{H}_i|^2 \, d\mu_t^i\, .
\end{equation}
Hence, Brakke's inequality \eqref{Brakke_inequ} passes to the limit. Moreover, 
by Theorem \ref{app_thm_easy_brakke} (epsilon-regularity) the convergence is smooth near points with Gaussian density close to $1$, so being unit-regular is preserved. Likewise, using again  \eqref{mass_bounds} it follows from a result of White that 
\begin{equation}
[V_{\mu_t^{i(t)}}]\to [V_{\mu_t}]
\end{equation}
as mod 2 flat chains, so being cyclic is also preserved.
\end{proof}

\begin{exercise}[monotonicity inequality]
Plug in the backwards heat kernel times a cutoff function as test function in the definition of Brakke flows and use the first variation formula for varifolds to check that the computation from the second lecture goes through, yielding 
\begin{equation}\label{app_loc_mon_brakke}
 \frac{d}{dt} \int\rho_{X_0}\chi^\rho_{X_0} d\mu_t \leq -\int \left|\vec{H}-\frac{(x-x_0)^\perp}{2(t-t_0)}\right|^2 \rho_{X_0}\chi^\rho_{X_0} d\mu_t\, .
\end{equation}
\end{exercise}

Finally, the notions of outer/inner flow and Brakke flow are compatible. Specifically, given any closed embedded initial surface $M_0$ using Ilmanen's elliptic regularization one can construct unit-regular, cyclic, integral Brakke flows $\mu_t$ and $\mu_t'$ with initial condition $\mathcal{H}^2\lfloor M_0$ such that the support is given by the closed sets $M_t$ and $M_t'$, respectively.
\bigskip

\begin{relatedPDEs}

Weak solutions for the harmonic map flow that satisfy the monotonicity inequality have been constructed by Chen-Struwe \cite{ChenStruwe}. Their proof shares some similarities with the construction of weak solutions for the Navier-Stokes equation by Leray \cite{Leray}. Constructing weak solutions for the Ricci flow is still a major open problem, for recent progress in this direction see \cite{KleinerLott_singular,HaslhoferNaber_Ricciflow,Sturm_SuperRicci,Bamler}.
\end{relatedPDEs}

\bigskip

\section{Flow through singularities}

In this lecture, we discuss mean curvature flow of surfaces through singularities. There is now a powerful and highly developed theory, thanks to several recent significant developments, including:
\begin{itemize}
\item the uniqueness result for cylindrical tangent flows
\item the classification of genus zero shrinkers
\item the proof of Ilmanen's mean-convex neighborhood conjecture
\item the proof of Huisken's genericity conjecture
\item the proof of Ilmanen's multiplicity-one conjecture
\end{itemize}

We will now discuss these results and their consequences in non-historical order. Throughout, we denote by $\mathcal{M}=\{M_t\}_{t\geq 0}$ the outer (or inner) flow starting at a closed embedded surface $M_0\subset\mathbb{R}^3$.\\

The capture the basic structure of singularities, given any space-time point $X_0=(x_0,t_0)$ one considers a tangent-flow at $X_0$, namely
\begin{equation}
\hat{\mathcal{M}}_{X_0}:=\lim_{i'\to \infty}\mathcal{D}_{\lambda_i}(\mathcal{M}-X_0)\ ,
\end{equation}
i.e. one shifts $X_0$ to the space-time origin, parabolically dilates by $\lambda_i\to \infty$, and passes to a subsequential limit. By Husiken's monotonicity formula and the compactness theorem for integral Brakke flows, tangent-flows always exit and are always self-similarly shrinking.

The tangent flows provided by the compactness theorem come with some integer multiplicity, which could a priori be bigger than one. This was ruled out in a recent breakthrough by Bamler-Kleiner.

\begin{theorem}[multiplicity-one]
If $M_0\subset \mathbb{R}^3$ is a closed embedded surface, then all tangent flows have multiplicity one.
\end{theorem}

The main potential scenario to rule out is two sheets connected by catenoidal necks that result in a multiplicity-two blowup limit. To do so, the authors introduce a novel sheet separation function $\mathfrak{s}$ (loosely speaking, if the sheets are graphs of $u_1$ and $u_2$, then $\mathfrak{s}\sim u_1-u_2$), and prove that away from the neck one has the differential inequality
\begin{equation}
(\partial_t-\Delta)\log \mathfrak{s}\geq 0.
\end{equation}
They then deal with the neck regions via delicate integral estimates.

Together with $\eps$-regularity and dimension reduction this immediately yields a sharp estimate for the size of the singular set $\mathcal{S}\subset \mathbb{R}^3\times \mathbb{R}_+$.

\begin{corollary}[partial regularity]
For the (outer or inner) flow of embedded surfaces the singular set $\mathcal{S}$ has  dimension at most 1.\footnote{Here, dimension refers to the Hausdorff (or Minkowski) dimension with respect to the parabolic metric on space-time, so e.g. $\mathrm{dim}(\mathbb{R}^3\times \mathbb{R}_+)=5$.}
\end{corollary}

It is know that the ends of noncompact shrinkers are either cylindrical or asymptotically conical, and already in the 90s it has been observed that mean curvature flow through conical singularities can be nonunique. However, asymptotically conical shrinkers can be excluded most of the time thanks to the following two fundamental results.

\begin{theorem}[genus zero shrinkers]
The only nontrivial shrinkers of genus zero are the round sphere and the round cylinder.
\end{theorem}

This result, due to Brendle, is proved by an ingenious application of the stability inequality and the maximum principle. 

\begin{theorem}[generic singularities]
For generic $M_0$, the only tangent flows at singular points are the round shrinking spheres and cylinders.
\end{theorem}

The result, due to Chodosh-Choi-Mantoulidis-Schulze, is proved most easily via a density drop argument, which shows that for any unstable singularity one can perturb such that $\Theta$ drops by a definite amount.

Observing also that spherical singularities are isolated, we can thus assume from now on that $\mathcal{M}$ has a neck-singularity at $X_0$, i.e.
\begin{equation}
\hat{\mathcal{M}}_{X_0}=\{S^1(\sqrt{2|t|})\times \mathbb{R} \}_{t<0}\, ,
\end{equation}
This implicitly makes use of the following foundational result.

\begin{theorem}[uniqueness of tangent flows]
Cylindrical tangent flows are unique, i.e. independent of the choice of rescaling factors  $\lambda_i\to \infty$.
\end{theorem}

This result, which in particular gives uniqueness of the axis, has been established via a Lojasiewicz-Simon argument  by Colding-Minicozzi, who overcame major difficulties caused by the noncompactness.

The key for mean curvature flow through neck-singularities, is the following result from my joint work with Choi and Hershkovits:

\begin{theorem}[mean-convex neighborhoods]
If $\mathcal{M}=\{M_t\}_{t\geq 0}$ has a neck-singularity at $(x_0,t_0)$, then there exists $\eps=\eps(x_0,t_0)>0$ such that $M_t\cap B_\eps(x_0)$ is mean-convex for $|t-t_0|<\eps$.
\end{theorem}

A major difficulty was to rule out the potential scenario of a degenerate neck-pinch with a non-convex cap. Here, to fully capture the singularity, one really wants to understand all limit flows,
\begin{equation}
\mathcal{M}^\infty:=\lim_{\lambda_i\to \infty}\mathcal{D}_{\lambda_i}(\mathcal{M}-X_i)\ ,
\end{equation}
where now $X_i\to X_0$ depends on $i$ (e.g. for the degenerate neck-pinch one chooses $X_i$ along the tip). While tangent-flows are always self-similarly shrinking by Huisken's monotonicity formula, limit flows can be much more general. A priori, choosing $\lambda_i\to \infty$ suitably, we only know that $\mathcal{M}^\infty$ is an ancient asymptotically cylindrical flow, namely an ancient, unit-regular, cyclic, integral Brakke flow $\mathcal{M}=\{ \mu_t\}_{t\in (-\infty,T_E)}$ whose tangent flow at $-\infty$ is a round shrinking cylinder, i.e.
\begin{equation}\label{eq_tang_}
\check{\mathcal{M}}:=\lim_{\lambda_i\to 0} \mathcal{D}_{\lambda_i}\mathcal{M} = \{S^1(\sqrt{2|t|})\times \mathbb{R} \}_{t<0}\, .
\end{equation}
Together with Choi and Hershkovits we classified all such flows:

\begin{theorem}[classification]\label{classification_theorem}
Any ancient asymptotically cylindrical flow in $\mathbb{R}^3$ is either a round shrinking cylinder, or a translating bowl soliton, or an ancient oval.
\end{theorem}

The classification is illustrated in Figure \ref{class_asympt_cyl}. In particular, note that all solutions appearing in the classification result are convex. Hence, the mean-convex neighborhood conjecture follows from the classification theorem via a short argument by contradiction.

\begin{figure}[H]
\includegraphics[width=10cm]{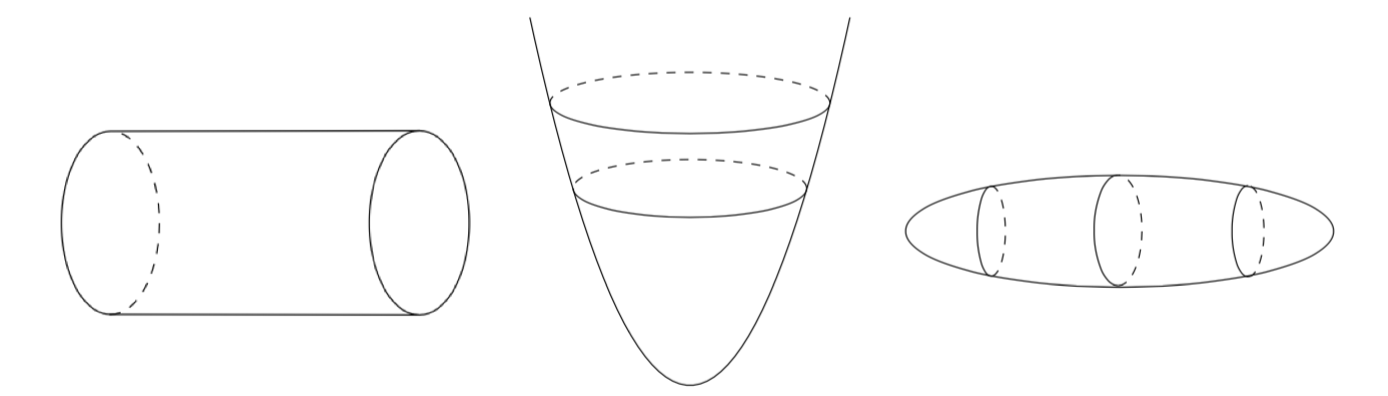}
\vspace{3mm}
\caption{Cylinder, bowl and ancient oval.}\label{class_asympt_cyl}
\end{figure}

As an application of the mean-convex neighborhood theorem and the multiplicity-one theorem one can confirm the uniqueness conjecture for mean curvature flow through neck-singularities:

\begin{corollary}[uniqueness]
Mean curvature flow through neck singularities is unique.
\end{corollary}

Indeed, it has been known since the 90s that mean-convex flows are unique. More recently, Hershkovits-White localized this result and showed that it is enough to assume that all singularities have a mean-convex neighborhood, and we exactly established this assumption.

As another application, taking also into account Brendle's classification of genus zero shrinkers, one can confirm a conjecture of White:

\begin{corollary}[flow of embedded 2-spheres]
Mean curvature flow of embedded 2-spheres is well posed.
\end{corollary}

To conclude this lecture, let us outline the main steps of our proof of the classification theorem (Theorem \ref{classification_theorem}). Given any ancient asymptotically cylindrical flow $\mathcal{M}$ that is not a round shrinking cylinder we have to show that it is either a bowl or an oval.

To get started, we set up a fine-neck analysis as follows. Given any $X_0=(x_0,t_0)$ we consider the renormalized flow
\begin{equation}
\bar{M}_\tau^{X_0}=e^{\tau/2}(M_{t_0-e^{-\tau}}-x_0)\, .
\end{equation}
Then, $\bar{M}_\tau^{X_0}$ converges for $\tau\to -\infty$ to $S^1(\sqrt{2})\times\mathbb{R}$. Hence, writing $\bar{M}_\tau^{X_0}$ locally as a graph of a function $u^{X_0}(z,\vartheta,\tau)$ over the cylinder the evolution is governed by the Ornstein-Uhlenbeck type operator
\begin{equation}
\mathcal{L}=\partial^2_z -\tfrac12 z \partial_z +\tfrac12 \partial_\vartheta^2+1\, .
\end{equation}
This operator has 4 unstable eigenfunctions, namely $1$, $\sin \vartheta$, $\cos\vartheta$, $z$, and 3 neutral eigenfunctions, namely $z\sin\vartheta$, $z\cos\vartheta$, $z^2-2$, and all other eigenfunctions are stable. By the Merle-Zaag ODE lemma for $\tau \to -\infty$ either the unstable or neutral eigenfunctions dominate.

If the unstable-mode is dominant, we prove that there exists a constant $a=a(\mathcal{M})\neq 0$ independent of the center point $X_0$ such that (after suitable recentering to kill the rotations) we have
\begin{equation}
u^{X_0}(z,\vartheta,\tau)=aze^{\tau/2}+o(e^{\tau/2})
\end{equation}
for all $\tau\ll 0$ depending only on the cylindrical scale of $X_0$. Moreover, we show that every point outside a ball of controlled size in fact lies on such a fine-neck. Hence, by the Brendle-Choi neck-improvement theorem the solution becomes very symmetric at infinity. Finally, we can apply a variant of the moving plane method to conclude that the solution is smooth and rotationally symmetric, and hence the bowl.

If the neutral-mode is dominant, then analyzing to ODE for the coefficient of the eigenfunction $z^2-2$ we show that there is an inwards quadratic bending, and consequently that the solution is compact. Blowing up near the tips, by the classification from the unstable-mode case we see bowls. Hence, using the maximum principle we can show that the flow is mean-convex and noncollapsed. Finally, we can apply the result by Angenent-Daskalopoulos-Sesum to conclude that it is an oval.

\bigskip

\begin{relatedPDEs}
Bamler-Kleiner recently proved uniqueness of 3d Ricci flow through singularities \cite{BamlerKleiner_uniqueness}. For 3d Ricci flow thanks to the Hamilton-Ivey pinching estimate one knows a priori that all blowup limits have positive curvature, so this is morally comparable to the flow of mean-convex surfaces. Motivated by our proof of the mean-convex neighborhood conjecture, it seems that there should be a canonical neighborhood theorem for neck-singularities in higher dimensional Ricci flow without assuming positive curvature a priori, c.f. \cite{Haslhofer_4dRicci}.
\end{relatedPDEs}

\bigskip

\section{Open problems}

In this final section, we discuss some of the most important open problems for the mean curvature flow of surfaces.

\begin{conjecture}[No cylinder conjecture]
The only complete embedded shrinker with a cylindrical end is the round cylinder.\footnote{I do not have any convincing heuristics why this should be true, so to err on the side of caution it might be better to call it a question instead of a conjecture.}
\end{conjecture}

To address this conjecture, one in particular has to rule out the scenario of shrinkers of mixed-type where some ends are cylindrical and some ends are conical. The statement becomes false if the completeness or embeddedness assumption is dropped. Together with prior work of Wang a resolution of the conjecture would imply the nice dichotomy that all singularities are either of conical-type or of neck-type.

\begin{conjecture}[Uniqueness of tangent-flows]
Tangent-flows are independent of the choice of sequence of rescaling factors $\lambda_i\to \infty$.
\end{conjecture}

By results of Schulze, Colding-Minicozzi and Chodosh-Schulze this would follow from a resolution of the no cylinder conjecture. Alternatively, one may also try to establish uniqueness directly.

\begin{conjecture}[Bounded diameter conjecture]
The intrinsic diameter stays uniformly bounded as one approaches the first singular time.
\end{conjecture}

This is motivated by a corresponding conjecture of Perelman for 3d Ricci flow. Du proved that the bounded diameter conjecture would follow from a resolution of the no cylinder conjecture. Alternatively, one may also try to establish the diameter bound directly.

Next, let me discuss several conjectures about the fine structure of the singular set. To be specific one can assume that the flow through singularities is either the outermost flow or the innermost flow.

\begin{conjecture}[Isolation conjecture]
All singularities are isolated unless an entire tube shrinks to a closed curve.\footnote{Recently, Sun and Xue proved that nondegenerate neck-singularities are isolated, which may be useful to establish a generic version of the isolation conjecture.}
\end{conjecture}

The conjecture is motivated by the principle that solutions of the level set flow, while only twice-differentiable, to some extent behave like analytic functions. Weaker formulations of the conjecture are that there are only finitely many singular times or that singularities are isolated generically. Another well-known related open problem is:

\begin{openproblem}[Self-similarity of blowup limits]
Are blowup limits always selfsimilar?
\end{openproblem}

In joint work with B. Choi and Hershkovits we proved that the ancient ovals occur as blowup limit if and only if there is an accumulation of spherical singularities. Yet another related open problem is:

\begin{openproblem}[Shrinking tubes]
Which closed curves can arise as singular set of a mean curvature flow of embedded surfaces?
\end{openproblem}

The only known example is the marriage-ring, which shrinks to a round circle.

Let us now switch gears and discuss two problems that specifically assume that the surface is topologically a two-sphere:

\begin{openproblem}[Mean curvature flow proof of Smale conjecture]
Is there a mean curvature flow proof of the Smale conjecture?\footnote{The $\pi_0$-part follows from \cite{BHH}, thanks to \cite{Brendle_sphere,CHH,DanielsHolgate,BK_mult1}, but showing that the higher homotopy groups also vanish requires more work.}
\end{openproblem}

In one of its many formulations, Smale's conjecture states that the space of embedded two-spheres in $\mathbb{R}^3$ is contractible. There is a geometric proof by Hatcher and a Ricci flow proof by Bamler-Kleiner, but it would be nice to have a direct proof by mean curvature flow.

\begin{conjecture}[Yau's conjecture]
The three-sphere equipped with any Riemannian metric contains at least $4$ embedded minimal two-spheres.\footnote{Recently, Wang and Zhou proved Yau's conjecture for generic metrics.}
\end{conjecture}

This is inspired by the classical work of Lusternik-Schnirelman and Grayson, which establishes the existence of at least 3 closed embedded geodesics on the two-sphere with any metric.

Finally, let me mention two somewhat more open ended problems:

\begin{openproblem}[Selection principle]
Is there a selection principle for flowing out of conical singularities?
\end{openproblem}

It has been proposed by Dirr-Luckhaus-Novaga and Yip that considering mean curvature flow with space-time white noise the physically relevant solutions will be selected in the vanishing noise limit.\footnote{By recent work of Chodosh, Daniels-Holgate and Schulze it suffices to study this problem for flows whose initial condition is exactly a cone.}

\begin{openproblem}[Immersed surfaces]
Develop a theory of weak solutions for the flow of immersed surfaces.
\end{openproblem}

Many of the methods described in this lecture series rely on embeddedness, so some fundamentally new ideas would be needed.

\bigskip

\begin{relatedPDEs}
Many of the problems discussed here, including in particular the uniqueness of tangent-cones/flows, genericity, optimal partial regularity, finiteness of singularities, self-similarity of blowup limits and selection principle for evolution through singularities, are of central importance in a wide range of partial differential equations.
\end{relatedPDEs}

\bigskip

\bibliography{lectures_mcf}

\newcommand{\noopsort}[1]{} \newcommand{\singleletter}[1]{#1}
\begin{thebibliography}{CHHW22}

\bibitem[ADS20]{ADS2}
S.~Angenent, P.~Daskalopoulos, and N.~Sesum.
\newblock Uniqueness of two-convex closed ancient solutions to the mean
  curvature flow.
\newblock {\em Ann. of Math. (2)}, 192(2):353--436, 2020.

\bibitem[AIC95]{ACI}
S.~Angenent, T.~Ilmanen, and D.~Chopp.
\newblock A computed example of nonuniqueness of mean curvature flow in {$\bold
  R^3$}.
\newblock {\em Comm. Partial Differential Equations}, 20(11-12):1937--1958,
  1995.

\bibitem[All72]{Allard}
W.~Allard.
\newblock On the first variation of a varifold.
\newblock {\em Ann. of Math. (2)}, 95:417--491, 1972.

\bibitem[And90]{Anderson}
M.~Anderson.
\newblock Convergence and rigidity of manifolds under {R}icci curvature bounds.
\newblock {\em Invent. Math.}, 102(2):429--445, 1990.

\bibitem[And12]{andrews1}
B.~Andrews.
\newblock Noncollapsing in mean-convex mean curvature flow.
\newblock {\em Geom. Topol.}, 16(3):1413--1418, 2012.

\bibitem[Bam23]{Bamler}
R.~Bamler.
\newblock Compactness theory of the space of super {R}icci flows.
\newblock {\em Invent. Math.}, 233(3):1121--1277, 2023.

\bibitem[BC19]{BC1}
S.~Brendle and K.~Choi.
\newblock Uniqueness of convex ancient solutions to mean curvature flow in
  {$\Bbb R^3$}.
\newblock {\em Invent. Math.}, 217(1):35--76, 2019.

\bibitem[BHH21]{BHH}
R.~Buzano, R.~Haslhofer, and O.~Hershkovits.
\newblock The moduli space of two-convex embedded spheres.
\newblock {\em J. Differential Geom.}, 118(2):189--221, 2021.

\bibitem[BK19]{BamlerKleiner_smale}
R.~Bamler and B.~Kleiner.
\newblock Ricci flow and contractibility of spaces of metrics.
\newblock {\em arXiv:1909.08710}, 2019.

\bibitem[BK22]{BamlerKleiner_uniqueness}
R.~Bamler and B.~Kleiner.
\newblock Uniqueness and stability of {R}icci flow through singularities.
\newblock {\em Acta Math.}, 228(1):1--215, 2022.

\bibitem[BK23]{BK_mult1}
R.~Bamler and B.~Kleiner.
\newblock On the multiplicity one conjecture for mean curvature flows of
  surfaces.
\newblock {\em arXiv:2312.02106}, 2023.

\bibitem[Bra78]{brakke}
K.~Brakke.
\newblock {\em The motion of a surface by its mean curvature}, volume~20 of
  {\em Mathematical Notes}.
\newblock Princeton University Press, Princeton, N.J., 1978.

\bibitem[Bre14]{Brendle_twopoint}
S.~Brendle.
\newblock Two-point functions and their applications in geometry.
\newblock {\em Bull. Amer. Math. Soc. (N.S.)}, 51(4):581--596, 2014.

\bibitem[Bre16]{Brendle_sphere}
S.~Brendle.
\newblock Embedded self-similar shrinkers of genus 0.
\newblock {\em Ann. of Math. (2)}, 183(2):715--728, 2016.

\bibitem[BW17]{BW}
J.~Bernstein and L.~Wang.
\newblock A topological property of asymptotically conical self-shrinkers of
  small entropy.
\newblock {\em Duke Math. J.}, 166(3):403--435, 2017.

\bibitem[CCMS20]{CCMS}
O.~Chodosh, K.~Choi, C.~Mantoulidis, and F.~Schulze.
\newblock Mean curvature flow with generic initial data.
\newblock {\em arXiv:2003.14344}, 2020.

\bibitem[CCMS21]{CCMS2}
O.~Chodosh, K.~Choi, C.~Mantoulidis, and F.~Schulze.
\newblock Mean curvature flow with generic low-entropy initial data.
\newblock {\em arXiv:2102.11978}, 2021.

\bibitem[CCS23]{CCS}
O.~Chodosh, K.~Choi, and F.~Schulze.
\newblock Mean curvature flow with generic initial data {II}.
\newblock {\em arXiv:2302.08409}, 2023.

\bibitem[CDHS23]{CDHS}
O.~Chodosh, J.~Daniels-Holgate, and F.~Schulze.
\newblock Mean curvature flow from conical singularities.
\newblock {\em arXiv:2312.00759}, 2023.

\bibitem[CGG91]{CGG}
Y.G. Chen, Y.~Giga, and S.~Goto.
\newblock Uniqueness and existence of viscosity solutions of generalized mean
  curvature flow equations.
\newblock {\em J. Differential Geom.}, 33(3):749--786, 1991.

\bibitem[CHH21]{CHH_ovals}
B.~Choi, R.~Haslhofer, and O.~Hershkovits.
\newblock A note on the selfsimilarity of limit flows.
\newblock {\em Proc. Amer. Math. Soc.}, 149(3):1239--1245, 2021.

\bibitem[CHH22]{CHH}
K.~Choi, R.~Haslhofer, and O.~Hershkovits.
\newblock Ancient low-entropy flows, mean-convex neighborhoods, and uniqueness.
\newblock {\em Acta Math.}, 228(2):217--301, 2022.

\bibitem[CHHW22]{CHHW}
K.~Choi, R.~Haslhofer, O.~Hershkovits, and B.~White.
\newblock Ancient asymptotically cylindrical flows and applications.
\newblock {\em Invent. Math.}, 229(1):139--241, 2022.

\bibitem[CIM15]{CIM}
T.~Colding, T.~Ilmanen, and W.~Minicozzi.
\newblock Rigidity of generic singularities of mean curvature flow.
\newblock {\em Publ. Math. Inst. Hautes \'{E}tudes Sci.}, 121:363--382, 2015.

\bibitem[CM12]{CM_generic}
T.~Colding and W.~Minicozzi.
\newblock Generic mean curvature flow {I}; generic singularities.
\newblock {\em Ann. of Math. (2)}, 175(2):755--833, 2012.

\bibitem[CM15]{CM_uniqueness}
T.~Colding and W.~Minicozzi.
\newblock Uniqueness of blowups and {L}ojasiewicz inequalities.
\newblock {\em Ann. of Math. (2)}, 182(1):221--285, 2015.

\bibitem[CM16]{CM_arrival}
T.~Colding and W.~Minicozzi.
\newblock Differentiability of the arrival time.
\newblock {\em Comm. Pure Appl. Math.}, 69(12):2349--2363, 2016.

\bibitem[CMS23]{CMS}
O.~Chodosh, C.~Mantoulidis, and F.~Schulze.
\newblock Mean curvature flow with generic low-entropy initial data {II}.
\newblock {\em arXiv:2309.03856}, 2023.

\bibitem[CS89]{ChenStruwe}
Y.~Chen and M.~Struwe.
\newblock Existence and partial regularity results for the heat flow for
  harmonic maps.
\newblock {\em Math. Z.}, 201(1):83--103, 1989.

\bibitem[CS21]{ChodoshSchulze}
O.~Chodosh and F.~Schulze.
\newblock Uniqueness of asymptotically conical tangent flows.
\newblock {\em Duke Math. J.}, 170(16):3601--3657, 2021.

\bibitem[DH22]{DanielsHolgate}
J.~Daniels-Holgate.
\newblock Approximation of mean curvature flow with generic singularities by
  smooth flows with surgery.
\newblock {\em Adv. Math.}, 410:Paper No. 108715, 42, 2022.

\bibitem[DLN01]{DLN}
N.~Dirr, S.~Luckhaus, and M.~Novaga.
\newblock A stochastic selection principle in case of fattening for curvature
  flow.
\newblock {\em Calc. Var. Partial Differential Equations}, 13(4):405--425,
  2001.

\bibitem[DPGS24]{DPGS}
Guido De~Philippis, Carlo Gasparetto, and Felix Schulze.
\newblock A {S}hort {P}roof of {A}llard's and {B}rakke's {R}egularity
  {T}heorems.
\newblock {\em Int. Math. Res. Not. IMRN}, (9):7594--7613, 2024.

\bibitem[Du21]{Du}
W.~Du.
\newblock Bounded diameter under mean curvature flow.
\newblock {\em J. Geom. Anal.}, 31(11):11114--11138, 2021.

\bibitem[Eck04]{Ecker_book}
K.~Ecker.
\newblock {\em Regularity theory for mean curvature flow}.
\newblock Progress in Nonlinear Differential Equations and their Applications,
  57. Birkh\"auser Boston Inc., Boston, MA, 2004.

\bibitem[ES91]{evans-spruck}
L.~Evans and J.~Spruck.
\newblock Motion of level sets by mean curvature. {I}.
\newblock {\em J. Differential Geom.}, 33(3):635--681, 1991.

\bibitem[GH20]{GH_diameter}
P.~Gianniotis and R.~Haslhofer.
\newblock Diameter and curvature control under mean curvature flow.
\newblock {\em Amer. J. Math.}, 142(6):1877--1896, 2020.

\bibitem[Gro99]{Gromov_mon}
M.~Gromov.
\newblock {\em Metric structures for {R}iemannian and non-{R}iemannian spaces},
  volume 152 of {\em Progress in Mathematics}.
\newblock Birkh\"{a}user Boston, Inc., Boston, MA, 1999.

\bibitem[Ham95]{Hamilton_survey}
R.~Hamilton.
\newblock The formation of singularities in the {R}icci flow.
\newblock In {\em Surveys in differential geometry, {V}ol. {II} ({C}ambridge,
  {MA}, 1993)}, pages 7--136. Int. Press, Cambridge, MA, 1995.

\bibitem[Has14]{Haslhofer_lectures_MCF}
R.~Haslhofer.
\newblock Lectures on mean curvature flow.
\newblock {\em arXiv:1406.7765}, 2014.

\bibitem[Has15]{Haslhofer_bowl}
R.~Haslhofer.
\newblock Uniqueness of the bowl soliton.
\newblock {\em Geom. Topol.}, 19(4):2393--2406, 2015.

\bibitem[Has16]{Haslhofer_lectures_CSF}
R.~Haslhofer.
\newblock Lectures on curve shortening flow.
\newblock {\em available at http://www.math.toronto.edu/roberth/}, 2016.

\bibitem[Has17]{Haslhofer_FieldsLectures}
R.~Haslhofer.
\newblock Introduction to mean curvature flow.
\newblock {\em videos available at
  https://video-archive.fields.utoronto.ca/list/event/1356}, 2017.

\bibitem[Has24]{Haslhofer_4dRicci}
R.~Haslhofer.
\newblock On $\kappa$-solutions and canonical neighborhoods in 4d {R}icci flow.
\newblock {\em J. Reine Angew. Math.}, 811:257--265, 2024.

\bibitem[Hat83]{Hatcher}
A.~Hatcher.
\newblock A proof of the {S}male conjecture, {${\rm Diff}(S^{3})\simeq {\rm
  O}(4)$}.
\newblock {\em Ann. of Math. (2)}, 117(3):553--607, 1983.

\bibitem[HI01]{HuiskenIlmanen}
G.~Huisken and T.~Ilmanen.
\newblock The inverse mean curvature flow and the {R}iemannian {P}enrose
  inequality.
\newblock {\em J. Differential Geom.}, 59(3):353--437, 2001.

\bibitem[HK17]{HK1}
R.~Haslhofer and B.~Kleiner.
\newblock Mean curvature flow of mean convex hypersurfaces.
\newblock {\em Comm. Pure Appl. Math.}, 70(3):511--546, 2017.

\bibitem[HK19]{HaslhoferKetover}
R.~Haslhofer and D.~Ketover.
\newblock Minimal 2-spheres in 3-spheres.
\newblock {\em Duke Math. J.}, 168(10):1929--1975, 2019.

\bibitem[HK23]{HaslhoferKetover2}
R.~Haslhofer and D.~Ketover.
\newblock Free boundary minimal disks in convex balls.
\newblock {\em arXiv:2307.01828}, 2023.

\bibitem[HN14]{HeinNaber}
H.-J. Hein and A.~Naber.
\newblock New logarithmic {S}obolev inequalities and an {$\epsilon$}-regularity
  theorem for the {R}icci flow.
\newblock {\em Comm. Pure Appl. Math.}, 67(9):1543--1561, 2014.

\bibitem[HN18]{HaslhoferNaber_Ricciflow}
R.~Haslhofer and A.~Naber.
\newblock Characterizations of the {R}icci flow.
\newblock {\em J. Eur. Math. Soc. (JEMS)}, 20(5):1269--1302, 2018.

\bibitem[HS99a]{huisken-sinestrari1}
G.~Huisken and C.~Sinestrari.
\newblock Mean curvature flow singularities for mean convex surfaces.
\newblock {\em Calc. Var. Partial Differential Equations}, 8(1):1--14,
  {\noopsort{a}}1999.

\bibitem[HS99b]{huisken-sinestrari2}
G.~Huisken and C.~Sinestrari.
\newblock Convexity estimates for mean curvature flow and singularities of mean
  convex surfaces.
\newblock {\em Acta Math.}, 183(1):45--70, {\noopsort{b}}1999.

\bibitem[HS09]{huisken-sinestrari3}
G.~Huisken and C.~Sinestrari.
\newblock Mean curvature flow with surgeries of two-convex hypersurfaces.
\newblock {\em Invent. Math.}, 175(1):137--221, 2009.

\bibitem[Hui84]{Huisken_convex}
G.~Huisken.
\newblock Flow by mean curvature of convex surfaces into spheres.
\newblock {\em J. Differential Geom.}, 20(1):237--266, 1984.

\bibitem[Hui90]{Huisken_monotonicity}
G.~Huisken.
\newblock Asymptotic behavior for singularities of the mean curvature flow.
\newblock {\em J. Differential Geom.}, 31(1):285--299, 1990.

\bibitem[HW20]{HershkovitsWhite_uniqueness}
O.~Hershkovits and B.~White.
\newblock Nonfattening of mean curvature flow at singularities of mean convex
  type.
\newblock {\em Comm. Pure Appl. Math.}, 73(3):558--580, 2020.

\bibitem[HW23]{HershkovitsWhite_set}
O.~Hershkovits and B.~White.
\newblock Avoidance for set-theoretic solutions of mean-curvature-type flows.
\newblock {\em Comm. Anal. Geom.}, 31(1):31--67, 2023.

\bibitem[Ilm93]{Ilmanen_set}
T.~Ilmanen.
\newblock The level-set flow on a manifold.
\newblock In {\em Differential geometry: partial differential equations on
  manifolds ({L}os {A}ngeles, {CA}, 1990)}, volume~54 of {\em Proc. Sympos.
  Pure Math.}, pages 193--204. Amer. Math. Soc., Providence, RI, 1993.

\bibitem[Ilm94]{Ilmanen}
T.~Ilmanen.
\newblock Elliptic regularization and partial regularity for motion by mean
  curvature.
\newblock {\em Mem. Amer. Math. Soc.}, 108(520):x+90, 1994.

\bibitem[Ilm95]{Ilmanen_mon}
T.~Ilmanen.
\newblock Singularities of mean curvature flow of surfaces.
\newblock {\em notes available at
  https://people.math.ethz.ch/{\textasciitilde}ilmanen/papers/sing.ps}, 1995.

\bibitem[Ilm03]{Ilmanen_problems}
T.~Ilmanen.
\newblock Problems in mean curvature flow.
\newblock {\em
  https://people.math.ethz.ch/{\textasciitilde}ilmanen/classes/eil03/problems03.ps},
  2003.

\bibitem[KL17]{KleinerLott_singular}
B.~Kleiner and J.~Lott.
\newblock Singular {R}icci flows {I}.
\newblock {\em Acta Math.}, 219(1):65--134, 2017.

\bibitem[KT14]{KasaiTonegawa}
K.~Kasai and Y.~Tonegawa.
\newblock A general regularity theory for weak mean curvature flow.
\newblock {\em Calc. Var. Partial Differential Equations}, 50(1-2):1--68, 2014.

\bibitem[Ler34]{Leray}
J.~Leray.
\newblock Sur le mouvement d'un liquide visqueux emplissant l'espace.
\newblock {\em Acta Math.}, 63(1):193--248, 1934.

\bibitem[LM23]{LiokumovichMaximo}
Y.~Liokumovich and D.~Maximo.
\newblock Waist inequality for 3-manifolds with positive scalar curvature.
\newblock {\em Perspectives in scalar curvature. {V}ol. 2}, pages 799--831,
  2023.

\bibitem[Man11]{Mantegazza_book}
C.~Mantegazza.
\newblock {\em Lecture notes on mean curvature flow}, volume 290 of {\em
  Progress in Mathematics}.
\newblock Birkh\"auser/Springer Basel AG, Basel, 2011.

\bibitem[Mul56]{Mullins}
W.~W. Mullins.
\newblock Two-dimensional motion of idealized grain boundaries.
\newblock {\em J. Appl. Phys.}, 27:900--904, 1956.

\bibitem[MZ98]{MZ}
F.~Merle and H.~Zaag.
\newblock Optimal estimates for blowup rate and behavior for nonlinear heat
  equations.
\newblock {\em Comm. Pure Appl. Math.}, 51(2):139--196, 1998.

\bibitem[OS88]{OsherSethian}
S.~Osher and J.~Sethian.
\newblock Fronts propagating with curvature-dependent speed: algorithms based
  on {H}amilton-{J}acobi formulations.
\newblock {\em J. Comput. Phys.}, 79(1):12--49, 1988.

\bibitem[Per02]{Perelman}
G.~Perelman.
\newblock The entropy formula for the {R}icci flow and its geometric
  applications.
\newblock {\em arXiv:math/0211159}, 2002.

\bibitem[Sch08]{Schulze_isoperimetric}
F.~Schulze.
\newblock Nonlinear evolution by mean curvature and isoperimetric inequalities.
\newblock {\em J. Differential Geom.}, 79(2):197--241, 2008.

\bibitem[Sch14]{Schulze_tangent}
F.~Schulze.
\newblock Uniqueness of compact tangent flows in mean curvature flow.
\newblock {\em J. Reine Angew. Math.}, 690:163--172, 2014.

\bibitem[Sch17]{Schulze_lectures}
F.~Schulze.
\newblock Introduction to mean curvature flow.
\newblock {\em Lecture notes available at https://www.felixschulze.eu}, 2017.

\bibitem[Smi82]{SimonSmith}
F.~Smith.
\newblock On the existence of embedded minimal 2-spheres in the 3-sphere,
  endowed with an arbitrary {R}iemannian metric.
\newblock {\em Phd thesis, Supervisor: Leon Simon, University of Melbourne},
  1982.

\bibitem[Str88]{Struwe_mon}
M.~Struwe.
\newblock On the evolution of harmonic maps in higher dimensions.
\newblock {\em J. Differential Geom.}, 28(3):485--502, 1988.

\bibitem[Stu18]{Sturm_SuperRicci}
K.-T. Sturm.
\newblock Super-{R}icci flows for metric measure spaces.
\newblock {\em J. Funct. Anal.}, 275(12):3504--3569, 2018.

\bibitem[SW09]{ShengWang}
W.~Sheng and X.~Wang.
\newblock Singularity profile in the mean curvature flow.
\newblock {\em Methods Appl. Anal.}, 16(2):139--155, 2009.

\bibitem[SW20]{SchulzeWhite}
F.~Schulze and B.~White.
\newblock A local regularity theorem for mean curvature flow with triple edges.
\newblock {\em J. Reine Angew. Math.}, 758:281--305, 2020.

\bibitem[SX22]{SunXue}
A.~Sun and J.~Xue.
\newblock Generic mean curvature flows with cylindrical singularities.
\newblock {\em arXiv:2210.00419}, 2022.

\bibitem[Ton19]{Tonegawa_book}
Y.~Tonegawa.
\newblock {\em Brakke's mean curvature flow}.
\newblock SpringerBriefs in Mathematics. Springer, Singapore, 2019.
\newblock An introduction.

\bibitem[Wan11]{Wang_convex}
X.~Wang.
\newblock Convex solutions to the mean curvature flow.
\newblock {\em Ann. of Math. (2)}, 173(3):1185--1239, 2011.

\bibitem[Wan16]{Wang_shrinker}
L.~Wang.
\newblock Asymptotic structure of self-shrinkers.
\newblock {\em arXiv:1610.04904}, 2016.

\bibitem[Whi00]{white_size}
B.~White.
\newblock The size of the singular set in mean curvature flow of mean convex
  sets.
\newblock {\em J. Amer. Math. Soc.}, 13(3):665--695, 2000.

\bibitem[Whi02]{White_ICM}
B.~White.
\newblock Evolution of curves and surfaces by mean curvature.
\newblock In {\em Proceedings of the {I}nternational {C}ongress of
  {M}athematicians, {V}ol. {I} ({B}eijing, 2002)}, pages 525--538. Higher Ed.
  Press, Beijing, 2002.

\bibitem[Whi03]{white_nature}
B.~White.
\newblock The nature of singularities in mean curvature flow of mean-convex
  sets.
\newblock {\em J. Amer. Math. Soc.}, 16(1):123--138, 2003.

\bibitem[Whi05]{white_regularity}
B.~White.
\newblock A local regularity theorem for mean curvature flow.
\newblock {\em Ann. of Math. (2)}, 161(3):1487--1519, 2005.

\bibitem[Whi09]{White_cyclic}
B.~White.
\newblock Currents and flat chains associated to varifolds, with an application
  to mean curvature flow.
\newblock {\em Duke Math. J.}, 148(1):41--62, 2009.

\bibitem[Whi15]{White_lectures}
B.~White.
\newblock Topics in mean curvature flow.
\newblock {\em Lecture notes by Otis Chodosh available at
  http://web.stanford.edu/~ochodosh/notes.html}, 2015.

\bibitem[WZ23]{WangZhou}
Z.~Wang and X.~Zhou.
\newblock Existence of four minimal spheres in $s^3$ with a bumpy metric.
\newblock {\em arXiv:2305.08755}, 2023.

\bibitem[Yip98]{Yip}
N.~Yip.
\newblock Stochastic motion by mean curvature.
\newblock {\em Arch. Rational Mech. Anal.}, 144(4):313--355, 1998.

\end{thebibliography}


\bigskip

\begin{references}
Mean curvature flow first appeared as a model for evolving interfaces in material science \cite{Mullins}. Its mathematical study was pioneered by Brakke \cite{brakke} and Huisken \cite{Huisken_convex}. Nice textbooks on mean curvature flow include the ones by Ecker \cite{Ecker_book}, Mantegazza \cite{Mantegazza_book} and Tonegawa \cite{Tonegawa_book}. I also recommend the notes from the lectures by White \cite{White_lectures} and Schulze \cite{Schulze_lectures}. Finally, let me point you to my own lecture notes \cite{Haslhofer_lectures_MCF,Haslhofer_lectures_CSF}, as well as the video recordings from my topics course at the Fields Institute \cite{Haslhofer_FieldsLectures}.
\end{references}

\bigskip

\begin{references}
The monotonicity formula for the mean curvature flow was discovered by Huisken \cite{Huisken_monotonicity}.
The local version from the above remark can be found in Ecker's book \cite{Ecker_book}. A generalization to weak solutions can be found in Ilmanen's notes \cite{Ilmanen_mon}. The epsilon-regularity theorem  for mean curvature flow was discovered by Brakke \cite{brakke}. The presented much simpler proof in the setting of smooth flows and limits thereof is due to White \cite{white_regularity}. A careful proof for general weak solutions was given in \cite{KasaiTonegawa} and \cite{DPGS}.
\end{references}

\bigskip

\begin{references}
Noncollapsing for mean-convex mean curvature flow was proved first by White \cite{white_size}. The notion of $\alpha$-noncollapsing was introduced by Sheng-Wang \cite{ShengWang}, and the beautiful maximum principle proof that it is preserved is due to Andrews \cite{andrews1}, see also \cite{Brendle_twopoint,Haslhofer_lectures_MCF} for expositions. The theory of mean-convex mean curvature flow has been established first in the fundamental work of White \cite{white_size,white_nature} and Huisken-Sinestrari \cite{huisken-sinestrari1,huisken-sinestrari2}. The streamlined treatment in the setting of $\alpha$-noncollapsed flows presented here is from my joint work with Kleiner \cite{HK1}. Finally, ancient $\alpha$-noncollapsed flows of surfaces have been classified in a recent breakthrough by Brendle-Choi \cite{BC1} and Angenent-Daskalopoulos-Sesum  \cite{ADS2}.
\end{references}

\bigskip

\begin{references}
The level-set flow was first studied numerically by Osher-Sethian \cite{OsherSethian}, and in the rigorous setting of viscosity solutions by Chen-Giga-Goto \cite{CGG} and Evans-Spruck \cite{evans-spruck}. The geometric reformulation presented here is due to Ilmanen \cite{Ilmanen_set}, see also \cite{HershkovitsWhite_set}. The fattening example with smooth closed initial condition is due to Ilmanen-White \cite{White_ICM}. The notion of discrepancy was introduced by Hershkovits-White \cite{HershkovitsWhite_uniqueness}. Evolving varifolds were introduced by Brakke \cite{brakke}. The compactness theorem for integral Brakke flows and existence via elliptic regularization can be found in Ilmanen's monograph \cite{Ilmanen},  see also \cite{HershkovitsWhite_uniqueness}. Finally, for the notions unit-regular and cyclic, and their preservation, see \cite{SchulzeWhite} and \cite{White_cyclic}.
\end{references}

\bigskip

\begin{references}
The multiplicity-one conjecture has been proved in a recent breakthrough by Bamler-Kleiner \cite{BK_mult1}. Nonuniqueness of flows through conical singularities has been observed in \cite{ACI,White_ICM}.
 Brendle's classification of genus zero shrinkers appeared in \cite{Brendle_sphere}. The genericity conjecture has been proved by Chodosh-Choi-Mantoulidis-Schulze, first via a classification of ancient one-sided flows in \cite{CCMS,CCS}, and later via a simpler density drop argument in \cite{CCMS2,CMS}. The work on generic flows has been pioneered by Colding-Minicozzi \cite{CM_generic}, who also proved the foundational uniqueness result for cylindrical tangent flows \cite{CM_uniqueness}, see also \cite{CIM}. The mean-convex neighborhood conjecture has been proved in my joint work with Choi and Hershkovits \cite{CHH}, and the reformulation in terms of ancient asymptotically cylindrical flows is from our follow up paper with White \cite{CHHW}. Important prior classification results can be found in \cite{Wang_convex,Haslhofer_bowl,BW,BC1,ADS2}. Finally, the Merle-Zaag ODE lemma is from \cite{MZ}, and the Brendle-Choi neck-improvement is from \cite{BC1}.
\end{references}

\bigskip

\begin{references}
Many of the questions discussed here can be found on Ilmanen's problem list \cite{Ilmanen_problems}. The ends of shrinkers are always either conical or cylindrical by \cite{Wang_shrinker}, see also \cite{BK_mult1}. Uniqueness of compact, cylindrical and asymptotically conical tangent-flows has been established in \cite{Schulze_tangent,CM_uniqueness,ChodoshSchulze}. Special cases of the bounded diameter conjecture have been proved in \cite{GH_diameter,Du}. The principle that solutions of the level set flow to some extent behave like analytic functions is due to Colding-Minicozzi \cite{CM_arrival}, and the recent proof that nondegenerate singularities are isolated can be found in \cite{SunXue}. The potential scenario of ancient ovals as limit flows has been investigated in \cite{CHH_ovals}. The references for the proofs of the Smale conjecture are \cite{Hatcher,BamlerKleiner_smale}. Yau's conjecture for generic metrics has been proven in \cite{WangZhou}, building on earlier work from \cite{SimonSmith,HaslhoferKetover}. Finally, the stochastic selection principle has been proposed in \cite{DLN,Yip}, and the reduction to exact conical initial conditions is thanks to \cite{CDHS}.
\end{references}

\bibliographystyle{alpha}

\bigskip

{\sc Department of Mathematics, University of Toronto, 40 St George Street, Toronto, ON  M5S 2E4, Canada}

\emph{E-mail:} roberth@math.toronto.edu

\end{document}